\documentclass[11pt,letterpaper,reqno]{amsart}
\usepackage[T1]{fontenc}
\usepackage{lmodern}
\usepackage[margin=1in]{geometry}
\usepackage{amsmath,amssymb,amsthm,mathtools,mathrsfs}
\usepackage{booktabs,microtype}
\usepackage[colorlinks=true,linkcolor=red,citecolor=red,urlcolor=blue,
  pdfauthor={Simon Becker and Benjamin Lawton},pdftitle={Counterexamples to the ordinary Fried conjecture near hyperbolic metrics},
  pdfsubject={Counterexamples to the ordinary Fried conjecture},
  pdfkeywords={Fried conjecture, Ruelle zeta function, dynamical torsion}]{hyperref}
\newtheorem{theorem}{Theorem}[section]
\newtheorem{lemma}[theorem]{Lemma}
\newtheorem{proposition}[theorem]{Proposition}
\theoremstyle{remark}

\newcommand{\Id}{\operatorname{Id}}

\newcommand{\tr}{\operatorname{tr}}
\newcommand{\str}{\operatorname{str}}

\newcommand{\Det}{\operatorname{Det}}
\newcommand{\ran}{\operatorname{ran}}
\newcommand{\vol}{\operatorname{vol}}
\newcommand{\cH}{\mathcal H}
\newcommand{\cU}{\mathcal U}
\newcommand{\cV}{\mathcal V}
\newcommand{\cC}{\mathcal C}
\title[Counterexamples to the ordinary Fried conjecture]
{Counterexamples to the ordinary Fried conjecture\newline near hyperbolic metrics}
\author{Simon Becker}
\email{simon.becker@unibocconi.it}
\address{Bocconi University, Milan, Italy}
\author{Benjamin Lawton}
\email{blawton@alumni.princeton.edu}
\address{Paraform, San Francisco, USA}

\begin{document}
\begin{abstract}
We give an LLM-assisted construction of counterexamples to the ordinary Fried
conjecture, with real analytic metrics arbitrarily close to a hyperbolic
metric in dimension three. For acyclic unitary
line characters tending to the trivial character, the modulus of
the Ruelle zeta function at zero divided by analytic torsion squared
tends to zero. 
\end{abstract}
\maketitle
\addtolength{\abovedisplayskip}{-1pt}
\addtolength{\belowdisplayskip}{-1pt}
\section{The statement}
We give a counterexample to a torsion identity conjectured by Fried
for real analytic negatively curved manifolds~\cite[pp.~180--181]{Fried95}.
We consider a closed oriented three-manifold $(\Sigma,g)$ with a real
analytic negatively curved metric and a unitary character
$\chi:\pi_1(\Sigma)\to U(1)$.
Write $H^k(\Sigma;\chi)$ for cohomology with coefficients in its flat
line bundle, and $H^\bullet$ for all degrees. Acyclicity means $H^k(\Sigma;\chi)=0$ for every $k.$
We write $\mathcal G_g^\#$ for the primitive oriented periodic orbits
of the geodesic flow, and
$\ell_g(\gamma)$ for the length of an orbit $\gamma$.
For such an orbit, $\chi(\gamma)$ is the holonomy along its projection
to $\Sigma$.
For a complex variable $s$, both orientations of every closed geodesic
occur in the product
\begin{equation}\label{Euler}
 Z_{g,\chi}(s)=\prod_{\gamma\in\mathcal G_g^\#}
       \bigl(1-\chi(\gamma)e^{-s\ell_g(\gamma)}\bigr),
 \qquad \operatorname{Re}s\gg1.
\end{equation}
We use the meromorphic continuation of this Ruelle zeta function
near zero~\cite{DGRS}.
We write $T_\chi(\Sigma)$ for Ray–Singer analytic torsion~\cite{RS}.
For acyclic unitary coefficients, it is independent of the metric.
In dimension three, Fried's formulation for acyclic unitary line
characters is
\begin{equation}\label{Friedidentity}
 |Z_{g,\chi}(0)|=T_\chi(\Sigma)^2.
\end{equation}
Zworski's Conjecture~9~\cite{Zw} states this identity for smooth, negatively curved metrics. 
See also \cite[Conjecture~21, (35)]{HKS} and
\cite[Introduction, (1.2)]{DGRS} for versions of this conjecture.

By the \emph{ordinary value}, we mean the value at $s=0$ of the
meromorphically continued zeta function when it is holomorphic there,
after the metric and character have been fixed. Chaubet--Dang~\cite{CD} introduced dynamical torsion, which combines
the leading Laurent coefficient at zero with the refined torsion of
the full generalized zero-resonance complex. They proved its
invariance under contact Anosov deformations and its holomorphic
dependence on acyclic representations. For a suitable Euler structure,
its ratio to Turaev torsion is locally constant on the acyclic
representation space. Our normalization argument uses their
finite-cut construction and deformation invariance. The distinction
between this corrected invariant and the ordinary zeta value is
essential here.

We write $\mathbf1$ for the trivial character. Our main result is:
\begin{theorem}\label{mainresult}
There exist a closed oriented hyperbolic three-manifold $(\Sigma,g_0)$,
real analytic negatively curved metrics $g_j\to g_0$ in $C^\infty$,
and acyclic unitary
line characters $\chi_j\to\mathbf1$ such that $Z_{g_j,\chi_j}$ is
holomorphic at zero, and
\begin{equation}\label{mainratio}
 \frac{|Z_{g_j,\chi_j}(0)|}{T_{\chi_j}(\Sigma)^2}\longrightarrow0.
\end{equation}
Thus \eqref{Friedidentity} fails for all sufficiently large $j$.
\end{theorem}

Ceki\'c--Delarue--Dyatlov--Paternain~\cite[Theorem~1]{CDDP}
proved that the order at zero of the untwisted Ruelle zeta function
changes from $4-2b_1(\Sigma)$ at a hyperbolic metric to
$4-b_1(\Sigma)$ under sufficiently small generic nonzero conformal
perturbations. Thus, when $b_1(\Sigma)>0$, this order depends on the
metric. Their result concerns trivial, hence nonacyclic, coefficients
on a three-dimensional base, whose geodesic flow acts on a
five-dimensional unit tangent bundle. Their description of the
generalized zero-resonant spaces and their conformal splitting
calculation provide the spectral input for our construction.
Building on these results, we construct a twisted
resonance crossing in a conformal metric family $g_\tau$ and a unitary
character family $\chi_\theta$. For each sufficiently small real
$\theta\ne0$, the selected degree-one resonance crosses zero at a
unique metric parameter $\tau=\sigma_\theta$.
As proved in Section~\ref{sec:normalization}, \eqref{Friedidentity}
holds at every other metric in this small family.

Our argument does not determine whether the Ruelle zeta function
vanishes at zero for the selected metrics. At each crossing, the generalized
zero-resonance space in horizontal degree two has dimension either two
or three. The zeta function is nonzero at zero in the former case and
has a simple zero there in the latter. The torsion-normalized value
tends to zero in either case, so the counterexample does not require
resolving this multiplicity.

The ordinary zeta value at zero need not depend continuously on the metric
parameter at the crossing: the common spectral factors are cancelled only
after the parameters have been fixed.
See \eqref{divisor} and \eqref{specialized}.
The acyclic neighborhoods in~\cite[Theorem~3]{DGRS} depend on the fixed
character. Our characters tend to the nonacyclic trivial character $\mathbf {1}$.

Section~\ref{sec:spectral} identifies the metrics and characters used in the
counterexample. Section~\ref{sec:normalization} compares the resulting zeta
values with analytic torsion and proves the theorem.
Appendix~\ref{app:inputs} provides the geometric construction and spectral
calculations.

\textbf{AI statement.}
This work was developed with substantial AI assistance.
S.B. initially prompted GPT-6.1 Sol
(Ultra reasoning, in Codex) with \emph{``Prove Conjecture 9''}
and attached~\cite{Zw}. The resulting argument led instead to the
counterexample studied here. The output referenced B.L.'s repository,
leading to the present collaboration. B.L. used Claude and
Claude Code for research, derivations, writing, checks, and Lean
formalization. AI assistance was also used in subsequent revisions in which the authors tried to rewrite the argument to make it more accessible.

The
\href{https://github.com/SimonLukas2604/Fried-s-conjecture/blob/0dcb387a3451c6d4d173ccc38b7c8d18c21ac462/original-gpt-outputs/holomorphic_continuity_crossing.tex}{initial output},
\href{https://github.com/SimonLukas2604/Fried-s-conjecture/blob/0dcb387a3451c6d4d173ccc38b7c8d18c21ac462/original-gpt-outputs/specialization_nilpotent_residue_audit_20261011.tex}{specialization audit},
and
\href{https://github.com/SimonLukas2604/Fried-s-conjecture/tree/0dcb387a3451c6d4d173ccc38b7c8d18c21ac462}{expanded argument and Lean files}
are available online, together with B.L.'s
\href{https://github.com/blawton/fried-conj-lean/blob/610d659ea15af869553f6e3e79c432880d9b0ab0/docs/counterexample_intuitive_view.tex}{disclosure}
and
\href{https://github.com/blawton/fried-conj-lean/blob/610d659ea15af869553f6e3e79c432880d9b0ab0/formalization.yaml}{formalization metadata}.
The Lean verification covers conditional algebraic and limiting
steps. The geometric and spectral inputs are not formalized.
Both authors directed the work, checked the sources and mathematical
content, and take responsibility for the content of the manuscript.

\noindent\textbf{Acknowledgment.} S.B. thanks Maciej Zworski for a
discussion that tangentially involved the Fried conjecture but led him to explore this question.
The authors acknowledge the foundational role of Chaubet
–Dang's dynamical torsion and variation theory~\cite{CD}, and of
Ceki\'c–Delarue–Dyatlov–Paternain's analysis of zero resonances
and their splitting under conformal metric perturbations~\cite{CDDP}.

\section{The spectral cut and the crossing}\label{sec:spectral}
Write $b_k(\Sigma)=\dim H^k(\Sigma;\mathbb R)$.
We choose $(\Sigma,g_0)$ with the first Betti number $b_1(\Sigma)=1$ and a
nonzero real harmonic one-form $\omega$. In Appendix~\ref{app:inputs},
we construct such a base and a real analytic function $b$ with the
required conformal splitting. For a complex character parameter $z$
and a small real metric parameter
$\tau\geq0$, we set
\[
 d_z=d+iz\omega\wedge,\qquad
 \chi_z(\gamma)=\exp\left(-iz\int_\gamma\omega\right),\qquad
 g_\tau=e^{-2\tau b}g_0.
\]
For small $\tau$, these metrics remain negatively curved, so their
geodesic flows are Anosov.
For every sufficiently small real $z=\theta\ne0$, $\chi_\theta$
is acyclic and unitary, as proved in Appendix~\ref{app:inputs}.
We work on the unit tangent bundle $M=S_{g_0}\Sigma$ of $g_0$.
The map $F_\tau(x,v)=(x,e^{\tau b(x)}v)$ identifies $M$ with the
unit tangent bundle of $g_\tau$. Write $\pi:M\to\Sigma$ for the projection.
On $M$, $d_z$ denotes $d+iz\pi^*\omega\wedge$.
We denote the pulled-back geodesic contact form and Reeb field by
$\alpha_\tau$ and $X_\tau$. For a vector field $Y$, we write $\iota_Y$ for
contraction by $Y$. In particular,
$\alpha_\tau(X_\tau)=1$ and $\iota_{X_\tau}d\alpha_\tau=0$.
We set
\[
 L_{z,\tau}=d_z\iota_{X_\tau}+\iota_{X_\tau}d_z,
 \qquad B_{z,\tau}=-L_{z,\tau}.
\]
We suppress parameters on operators when clear. Resonances in the
spectral coordinate $s$ are poles of the continued resolvent. On the finite
spaces used below, they are eigenvalues of $B$, since $(L+s\Id)^{-1}=(s\Id-B)^{-1}.$

Choose a fixed spectral cut: the positively oriented circle
$\gamma_{\rm cut}=\{|s|=\lambda_{\rm cut}\}$, with $0<\lambda_{\rm cut}<1$.
It encloses only the zero resonance at $(z,\tau)=(0,0)$ and has no
boundary resonance in any degree. For $0\leq k\leq4$, we write
$\cU^k_{z,\tau}$ for the
generalized resonance space inside this circle on horizontal
degree-$k$ forms $u$, meaning $\iota_{X_\tau}u=0$.
On the generalized space associated with an eigenvalue $\lambda$,
some power of $B-\lambda\Id$ vanishes.
Every degree-$k$ form has a unique decomposition into a horizontal
part $u$ and a part containing $\alpha_\tau$:
\[
 u+\alpha_\tau\wedge v,\qquad
 \iota_{X_\tau}u=\iota_{X_\tau}v=0.
\]
Here $v$ has degree $k-1$. We use \emph{full forms} to refer to the entire space, including both summands.
The ranks of the entire fixed cluster on horizontal forms in degrees
$0,\ldots,4$, and on full forms in degrees $0,\ldots,5$, are
\begin{equation}\label{ranks}
 (1,2,4,2,1),\qquad (1,3,6,6,3,1).
\end{equation}
These ranks count algebraic multiplicities, including Jordan blocks.
The initial ranks are verified in
Appendix~\ref{app:inputs}. Lemma~\ref{holom} proves that they persist.
We define
\[
 A_k=B|_{\cU^k},\qquad c=\tr A_0,\qquad t=\tr A_1-c,
 \qquad P_k(s,z,\tau)=\det(s\Id-A_k).
\]
We call the matrices representing $B$ on these finite-dimensional
resonance spaces, in locally chosen bases, the cut matrices. We use
the same terminology for the full-form cluster. Their characteristic
polynomials are independent of the choice of bases.
Since $\cU^0$ has dimension one, $A_0=c\Id$, so $c$ is the scalar
resonance in the cut. By~\eqref{P1}, $c$ and $t$ are the two
degree-one roots, counted with multiplicity.
For a real character parameter $\theta$, we write $d_\theta^*$ for the
$L^2(g_0)$ adjoint of $d_\theta$ and
\[
 \Delta_{\mathrm H,\theta}^{(k)}=d_\theta^*d_\theta+d_\theta d_\theta^*
\]
for the twisted Hodge Laplacian on degree-$k$ forms on $\Sigma$.
We denote its small scalar eigenvalue by $\mu_0(\theta)$ and set $c_0(\theta):=-1+\sqrt{1-\mu_0(\theta)},$
with the square root equal to one at the origin. The $L^2$ norm uses $g_0$, and
$\vol(\Sigma)$ denotes its volume. We set
$\kappa=\|\omega\|_{L^2(\Sigma)}^2/\vol(\Sigma)>0$.
For small real $\theta\ne0$, the calculations in Appendix~\ref{app:inputs},
in particular \eqref{app:scalarquadratic} and
\eqref{app:clusterpolynomials}, give
\begin{align}\label{hyper}
 c(\theta,0)=t(\theta,0)&=c_0(\theta)<0,
 &P_2(s,\theta,0)&=(s-c_0)^3(s+c_0),\\
 \mu_0(\theta)&=\kappa\theta^2+O(\theta^4),
 &c_0(\theta)&=-\tfrac\kappa2\theta^2+O(\theta^4).\label{quadratic}
\end{align}
The conformal variation computed in Appendix~\ref{app:inputs} gives
\[
 r:=\left.\frac{d}{d\tau}\tr A_1(0,\tau)\right|_{\tau=0}>0.
\]
Since $c(0,\tau)=0$, this is also $\partial_\tau t(0,0)$, and
\begin{equation}\label{taxis}
 c(0,\tau)=0,\qquad t(0,\tau)=r\tau+o(\tau).
\end{equation}

\begin{lemma}\label{holom}
There exist $\rho,\varepsilon>0$ such that, for $|z|<\rho$ and
$\tau\in[0,\varepsilon]$, the horizontal resonance spaces
$\cU^k_{z,\tau}$, $0\leq k\leq4$, and the full resonance spaces $\cU^k_{z,\tau}\oplus\alpha_\tau\wedge\cU^{k-1}_{z,\tau},$ with $
 0\leq k\leq5,$
inside $\gamma_{\rm cut}$ have the ranks in~\eqref{ranks}.
Here $\cU^{-1}_{z,\tau}=\cU^5_{z,\tau}=0$.

Bases of these spaces can be chosen so that the matrices representing
the restrictions of $B_{z,\tau}$ have entries holomorphic in $z$ and
$C^1$ in $\tau$. The entries and their $\tau$-derivatives are jointly
continuous on
$\{|z|\leq\rho'\}\times[0,\varepsilon]$ for every $0<\rho'<\rho$.
The $\tau$-derivatives are also holomorphic in $z$.

The coefficients of the characteristic polynomials of these
restrictions have the same regularity. For the horizontal restriction
$A_k(z,\tau)=B_{z,\tau}|_{\cU^k_{z,\tau}}$, the polynomial is
$P_k(s,z,\tau)=\det(s\Id-A_k(z,\tau))$.
Each characteristic coefficient $a(z,\tau)$ satisfies $a(-z,\tau)=a(z,\tau)$ and $
 a(\bar z,\tau)=\overline{a(z,\tau)}.$
In particular, these coefficients are real for real $z$.
\end{lemma}
\begin{proof}
We take the common spaces $\cH_0,\ldots,\cH_4$ from
\cite[Appendix B.2--B.6]{CD}, on which the continued resolvent is bounded.
Their inclusions are continuous:
\[
 \cH_4\hookrightarrow\cH_3\hookrightarrow\cH_2
       \hookrightarrow\cH_1\hookrightarrow\cH_0.
\]
$\cH_{j+1}$ has one more order of Sobolev regularity than $\cH_j$.
In each degree the untwisted resolvents
$R_{0,\tau}(s)=(L_{0,\tau}+s)^{-1}$ are uniformly bounded on the
contour at every level. The twist is a bounded multiplier:
\[
 L_{z,\tau}=L_{0,\tau}+zV_\tau,
 \qquad V_\tau=i\pi^*\omega(X_\tau)
 \quad\text{\cite[(8.9)]{CD}}.
\]
We choose one $\rho$ so that $\|zV_\tau R_{0,\tau}(s)\|\leq1/2$
on all these levels, uniformly in $\tau$ and $s\in\gamma_{\rm cut}$.
We write $R_{z,\tau}(s)=(L_{z,\tau}+s\Id)^{-1}$ for the twisted
resolvent. The norm bound gives the convergent series
$(\Id+zV_\tau R_{0,\tau}(s))^{-1}
 =\sum_{n=0}^{\infty}(-zV_\tau R_{0,\tau}(s))^n.$
This gives the resolvent formula:
\begin{equation}\label{Neumann}
 R_{z,\tau}(s)=R_{0,\tau}(s)
                  (\Id+zV_\tau R_{0,\tau}(s))^{-1}
\end{equation}
The resolvent is holomorphic throughout the chosen disk in the character parameter. For two nearby real metric
parameters $\tau$ and $\tau'$, the identity
\[
 R_{z,\tau}-R_{z,\tau'}
 =R_{z,\tau}(L_{z,\tau'}-L_{z,\tau})R_{z,\tau'}
\]
gives continuity $\cH_{j+1}\to\cH_j$ for $0\leq j\leq3$,
locally uniformly in $z$.
For differentiability, \cite[Appendix B.5, (B.6)]{CD} gives operator bounds
between spaces whose Sobolev regularity differs by three orders.
On each chosen space, the multiplier $V_\tau$ depends smoothly on $\tau$. Applying the same resolvent identity to difference
quotients gives
\[
 \frac{R_{z,\tau'}-R_{z,\tau}}{\tau'-\tau}
 =-R_{z,\tau'}
     \frac{L_{z,\tau'}-L_{z,\tau}}{\tau'-\tau}R_{z,\tau}.
\]
The difference quotient of $L_{z,\tau}$ converges in the operator norm
for maps from $\cH_{j+1}$ to $\cH_j$. The resolvents are also
continuous in the operator norm for maps from $\cH_{j+1}$ to $\cH_j$.
We can therefore pass to the limit in the product as an operator from
$\cH_3$ to $\cH_0$, obtaining
\begin{equation}\label{metricresolvent}
 \partial_\tau R_{z,\tau}(s)
 =-R_{z,\tau}(s)(\partial_\tau L_{0,\tau}
                    +z\partial_\tau V_\tau)R_{z,\tau}(s)
 \quad\text{as maps }\cH_3\to\cH_0.
\end{equation}
The derivative is jointly continuous and locally uniformly holomorphic
in $z$.
Consequently the full-cut projection $\Pi_{z,\tau}$ and its first
moment $\mathsf T_{z,\tau}$, defined by
\begin{equation}\label{xp:projector-moment}
 \begin{aligned}
 \Pi_{z,\tau}&=\frac1{2\pi i}
     \int_{\gamma_{\rm cut}} R_{z,\tau}(s)\,ds \text{ and }
 \mathsf T_{z,\tau}=\frac1{2\pi i}
     \int_{\gamma_{\rm cut}} sR_{z,\tau}(s)\,ds
     =B_{z,\tau}\Pi_{z,\tau},
 \end{aligned}
\end{equation}
have the same regularity. Untwisted ranks persist by
\cite[Appendix B.6]{CD}. To check rank persistence under twisting,
we write $\|\cdot\|_j$ for the operator norm on $\cH_j$ and choose
constants $K_R,K_V>0$ such that, on every chosen level,
\[
 \|R_{0,\tau}(s)\|_j\leq K_R,
 \qquad\|V_\tau R_{0,\tau}(s)\|_j\leq K_V.
\]
The series and the contour length $2\pi\lambda_{\rm cut}$ give
\[
 \|\Pi_{z,\tau}-\Pi_{0,\tau}\|_j
 \leq\lambda_{\rm cut}K_R\frac{|z|K_V}{1-|z|K_V}<1,
\]
after shrinking $\rho$ again. Two projections at distance less than
one have the same finite rank.
Let $N$ be the dimension of the full resonance space in the chosen degree. Choose fixed smooth
forms $f_1,\ldots,f_N$ and linear functionals
$\varphi_1,\ldots,\varphi_N$ given by pairing with fixed smooth dual
forms, so that the matrix $F$ below is invertible at the origin:
\begin{equation}\label{xp:frame-matrices}
 F_{ij}=\varphi_i(\Pi f_j),\qquad (M_B)_{ij}=\varphi_i(\mathsf T f_j)
\end{equation}
The vectors $\Pi f_j$ then form a basis of the cut near the origin.
The matrix of $B$ in this basis is $F^{-1}M_B$, because
\[
 B\Pi f_j=\sum_{l=1}^N\Pi f_l(F^{-1}M_B)_{lj}.
\]
We write $N_h$ for the horizontal rank in that degree. For the horizontal
cut, we repeat this construction with the projection
$(\Id-\alpha_\tau\wedge\iota_{X_\tau})\Pi$ and with $N_h$ smooth
forms $f_j$ and $N_h$ linear functionals $\varphi_i$, chosen as above.
After shrinking $\rho,\varepsilon$, the matrices used to construct the full
and horizontal bases remain invertible throughout the parameter set. Fix $0<\rho_*<\rho$ and put
$\mathcal P=\{(z,\tau):|z|\le\rho_*,\ 0\le\tau\le\varepsilon\}$.
The inverse frame matrices are uniformly bounded on $\mathcal P$.
The horizontal projection commutes with $\Pi$. We set
$\cU^{-1}=\cU^5=0$. The contact splitting gives
\[
 \dim\ran(\Pi|_{\text{degree }k})
 =\dim\cU^{k-1}+\dim\cU^k\qquad(0\leq k\leq5).
\]
Since the full resonance spaces have constant dimensions, this identity
shows by induction on $k$, starting with $\cU^{-1}=0$, that the horizontal
dimensions are constant as well. In the bases constructed above, the
matrices representing $B$ have entries holomorphic in $z$ and $C^1$ in $\tau$.

We denote an orbit's transverse symplectic return map by $P_\gamma$,
its length by $\ell_\gamma=\ell_{g_\tau}(\gamma)$, and the integral
of the harmonic form along its projection by
$I_\gamma=\int_{\pi\gamma}\omega$.
The degree-$k$ dynamical determinant in a right half-plane is
\[
 \mathscr D_k(s,z,\tau)=\exp\left(-\sum_{\gamma\in\mathcal G^\#_{g_\tau}}
 \sum_{j\geq1}\frac{e^{-sj\ell_\gamma-izjI_\gamma}}{j}
 \frac{\tr(\Lambda^kP_\gamma^j)}{|\det(\Id-P_\gamma^j)|}\right).
\]
The reversed orbit $\bar\gamma$ has
$I_{\bar\gamma}=-I_\gamma$ and return map
similar to $P_\gamma^{-1}$. A symplectic map preserves the transverse
two-form. In particular, its determinant is one and its inverse is
similar to its transpose. Thus
\[
 \tr(\Lambda^kP_\gamma^{-j})=\tr(\Lambda^kP_\gamma^j),\qquad
 |\det(\Id-P_\gamma^{-j})|=|\det(\Id-P_\gamma^j)|.
\]
Pairing the two orientations gives $e^{-izjI_\gamma}+e^{izjI_\gamma}=2\cos(zjI_\gamma).$
Hence the determinant satisfies
\[
 \mathscr D_k(s,-z,\tau)=\mathscr D_k(s,z,\tau),\qquad
 \overline{\mathscr D_k(s,z,\tau)}
   =\mathscr D_k(\bar s,\bar z,\tau).
\]
These identities continue meromorphically. Apply
\cite[Section~5, Theorem~4]{DG} to the complex bundle
$\Lambda^k\operatorname{Ann}(X_\tau)\otimes L_{\chi_z}$ with lift $L_{z,\tau}$.
See also \cite[Remark~3]{DGRS}. Here $L_{\chi_z}$ is the flat character line.
The logarithmic derivative has only simple poles, with residues equal to
the algebraic resonance multiplicities. Integrating shows that these
multiplicities are the zero orders of the determinant, also for complex $z$
and perturbed metrics.
The zeros inside the contour therefore determine the monic cut
polynomial, including multiplicities, and give the asserted symmetries.
The full-form characteristic polynomials are products of those in
the two adjacent horizontal degrees, so they have the same symmetries.
\end{proof}

\begin{proposition}\label{prop:crossing}
After shrinking $\rho_*,\varepsilon$, the cut polynomials factor on
$\mathcal P$ as
\begin{align}
 P_1(s)&=(s-c)(s-t),\label{P1}\\
 P_2(s)&=(s-c)(s-t)Q(s),\qquad Q(s)=s^2-e_1s+e_2.\label{P2}
\end{align}
The coefficients $e_1,e_2$ are even and holomorphic in $z$, $C^1$ in
$\tau$, and real for real $z$. They are the sum and product of the
two remaining middle-degree roots, counted with algebraic multiplicity.
Moreover,
\begin{equation}\label{twoaxisfactor}
 e_1(z,\tau)=z^2\tau E(z,\tau),
\end{equation}
where $E$ is continuous and bounded on $\mathcal P$.
For every sufficiently small real $\theta\ne0$, the function
$\tau\mapsto t(\theta,\tau)$ has a unique zero
$\sigma_\theta\in(0,\varepsilon)$, and
\begin{equation}\label{cross}
 t(\theta,\sigma_\theta)=0,\qquad
 \frac{\sigma_\theta}{\theta^2}\longrightarrow\frac\kappa{2r}
 \quad(\theta\to0).
\end{equation}
At this crossing, $c(\theta,\sigma_\theta)<0$ and
$e_1(\theta,\sigma_\theta)=O(\theta^4)$.
Writing $V^k\subset\cU^k_{\theta,\sigma_\theta}$ for the horizontal
generalized zero-resonance space, we have $(\dim V^0,\ldots,\dim V^4)=(0,1,m,1,0),$ for $ m\in\{2,3\},$
and, with all coefficients evaluated at $(\theta,\sigma_\theta)$,
\begin{equation}\label{two}
 e_2=0,\qquad m=3\Longleftrightarrow e_1=0,
 \qquad m=2\Longleftrightarrow e_1\ne0.
\end{equation}
\end{proposition}
\begin{proof}
We define operators on horizontal forms by $D=d_z-\alpha_\tau\wedge L$ and $ W=d\alpha_\tau\wedge.$
They satisfy
\begin{equation}\label{Cartan}
 [B,D]=0,\qquad D^2=WB.
\end{equation}
Here $D$ is the horizontal part of $d_z$. It does not in general
square to zero. Indeed, for a horizontal form $u$,
\[
 d_zu=Du+\alpha_\tau\wedge Lu,\qquad
 0=d_z^2u=D^2u+WLu,
\]
where the terms containing $\alpha_\tau$ cancel because $D$ and $L$
commute. Thus $D^2u=0$ on zero eigenvectors of $B$.
The two-form $d\alpha_\tau$ is nondegenerate on the four-dimensional
contact plane $\ker\alpha_\tau$. Its wedge maps, called the contact
Lefschetz maps, commute with $B$ and give isomorphisms
\[
 W:\cU^1\longrightarrow\cU^3,\qquad
 W^2:\cU^0\longrightarrow\cU^4.
\]
For small real $\theta\ne0$, pullback acyclicity gives
$H^0(M;\pi^*\chi_\theta)=0$. Hence \cite[Lemma~7.4]{DGRS}
excludes a scalar resonance at zero, so $c\ne0$. For a nonzero scalar
eigenvector $f$ with $Bf=cf$, we have
\[
 D^2f=cWf\ne0,\qquad Df\ne0,\qquad B(Df)=cDf.
\]
Hence $c$ is also a degree-one root. Since the degree-one space has
dimension two and trace $c+t$, this proves \eqref{P1}, first for real
nonzero characters and then by holomorphy everywhere.
Using evenness, \eqref{taxis}, and Cauchy's formula, we define quotient
functions $C_c$ and $C_t$ by
\begin{equation}\label{ctfactors}
 c=z^2C_c,\qquad t=t(0,\tau)+z^2C_t,
 \qquad C_c(0,0)=C_t(0,0)=-\kappa/2.
\end{equation}
Here and below the quotients are holomorphic in $z$ and $C^1$ in
$\tau$, locally uniformly in $z$. For any even such function $a_0$ with
$a_0(0,\tau)=0$, and fixed $\rho_*<a<\rho$,
\begin{equation}\label{Cauchyquotient}
 \frac{a_0(z,\tau)}{z^2}
 =\frac1{2\pi i}\int_{|w|=a}
       \frac{a_0(w,\tau)}{w^2(w-z)}\,dw,\qquad |z|<a.
\end{equation}
Differentiating under the integral proves the stated regularity of the
extended quotient and its derivative with respect to the metric parameter. Shrink $\rho_*,\varepsilon$ so that $C_c$ is
nonzero on $\mathcal P$.

When $A_1$ is invertible, $D:\cU^1\to\cU^2$ is injective. Indeed,
\[
 Du=0\ \Longrightarrow\ D^2u=WBu=0
 \ \Longrightarrow\ Bu=0\ \Longrightarrow\ u=0.
\]
Here the second implication uses the isomorphism $W$, and the last
uses invertibility of $A_1$. Since $BD=DB$, the restriction of $B$ to
$D(\cU^1)$ has characteristic polynomial $P_1$. It follows that $P_1$
divides $P_2$.
The remainder in monic polynomial division is holomorphic in $z$
and $C^1$ in $\tau$. For fixed small $\tau>0$, $t(0,\tau)\ne0$,
and \eqref{ctfactors} shows that $c(\cdot,\tau)$ is not identically zero.
Thus $A_1$ is invertible on a dense set of $z$. At $\tau=0$, the same
conclusion follows from \eqref{hyper}. The remainder vanishes, giving \eqref{P2}. The quotient
coefficients inherit the regularity and symmetries of Lemma~\ref{holom}.
Shrink $\varepsilon$ below the untwisted splitting bound $\varepsilon_0$
established in Appendix~\ref{app:inputs}.
The hyperbolic identities and the untwisted algebraic multiplicities at zero
in degrees one and two from~\cite[Theorem~3 and Corollary~4.1]{CDDP}
give, respectively,
\begin{equation}\label{qhyper}
 e_1(z,0)=0,\quad e_2(z,0)=-c(z,0)^2,
 \qquad e_1(0,\tau)=e_2(0,\tau)=0.
\end{equation}
For the hyperbolic axis, division of \eqref{hyper} by \eqref{P1} gives
\[
 Q(s,z,0)=(s-c(z,0))(s+c(z,0))=s^2-c(z,0)^2.
\]
For $z=0$ and small $\tau>0$, the horizontal generalized
zero-resonance dimensions are $(1,1,3,1,1)$.
The map $D$ sends eigenvectors for the remaining degree-one root
$t(0,\tau)\ne0$ to nonzero degree-two eigenvectors with the same eigenvalue. The ranks therefore give
\[
 P_1(s,0,\tau)=s(s-t(0,\tau)),\quad
 P_2(s,0,\tau)=s^3(s-t(0,\tau)),\quad Q(s,0,\tau)=s^2.
\]
Continuity gives the same coefficient identities at $\tau=0$.
Applying \eqref{Cauchyquotient} to $e_1$ gives
\begin{equation}\label{factors}
 e_1=z^2A_e,\qquad A_e(z,0)=0.
\end{equation}
The fundamental theorem of calculus, using $A_e(z,0)=0$, gives the
two-axis factorization \eqref{twoaxisfactor} of~\cite{Lawton}, with
$E(z,\tau)=\int_0^1\partial_\tau A_e(z,v\tau)\,dv$ continuous and
bounded on $\mathcal P$.

For real $\theta$ and $\tau\geq0$, equations \eqref{taxis} and
\eqref{ctfactors} give
\begin{equation}\label{crossingbalance}
 t(\theta,\tau)=r\tau-\frac\kappa2\theta^2
                   +o(\tau+\theta^2)
\end{equation}
as $(\theta,\tau)\to(0,0)$.
By Lemma~\ref{holom}, shrink $\rho_*,\varepsilon$ so that
$\partial_\tau t(\theta,\tau)>r/2$ for real $(\theta,\tau)\in\mathcal P$.
Thus $t(0,\tau)\geq r\tau/2$.
Choose $K_*>0$ with $|C_t|\leq K_*$ on $\mathcal P$
and restrict $\theta$ so that $4K_*\theta^2/r<\varepsilon$. Hence
\[
 t(\theta,0)=c_0(\theta)<0,
 \qquad t(\theta,4K_*\theta^2/r)\geq K_*\theta^2>0.
\]
Continuity and strict monotonicity give a unique zero $\sigma_\theta$ of
$\tau\mapsto t(\theta,\tau)$, with
$0<\sigma_\theta\leq4K_*\theta^2/r$.
Using this bound in~\eqref{crossingbalance} gives $0=r\frac{\sigma_\theta}{\theta^2}
       -\frac\kappa2+o(1).$ This proves \eqref{cross}.
Equation \eqref{twoaxisfactor} now gives
$e_1(\theta,\sigma_\theta)=O(\theta^4)$.
The scalar root $c$ stays negative, since $c=\theta^2C_c$ and
$C_c\to-\kappa/2$.

Fix a crossing and suppress the parameters. Since $c\ne0$, we have
$V^0=V^4=0$, and the degree-one zero is simple. Lefschetz gives
$\dim V^1=\dim V^3=1$ and $L|_{V^1}=L|_{V^3}=0$.
Set $m=\dim V^2$.
We take $\alpha=\alpha_{\sigma_\theta}$ and set $V^{-1}=V^5=0$.
The full generalized zero-resonance complex has spaces
\[
 \cV^k=V^k\oplus\alpha\wedge V^{k-1}\qquad(0\leq k\leq5).
\]
For $u\in V^k$ and $v\in V^{k-1}$, its differential is
\begin{equation}\label{xp:full-zero-complex}
 d_\theta(u+\alpha\wedge v)
   =Du+Wv+\alpha\wedge(Lu-Dv).
\end{equation}
The resonant de Rham theorem~\cite[Theorem 2.1]{DR} identifies the
cohomology of this full generalized zero-resonance complex with the
twisted de Rham cohomology of $M$. Pullback acyclicity therefore
makes the full complex exact. Exactness does not require its spaces
to vanish, nor does it imply that $L$ vanishes on generalized zero
states. We write $D_k$
for $D|_{V^k}$. Since $V^0=V^4=0$, its end terms are
$\cV^1=V^1$ and $\cV^4=\alpha\wedge V^3$. For
$u\in V^3$ and $v\in V^2$, \eqref{xp:full-zero-complex} reduces to $d_\theta(u+\alpha\wedge v)=-\alpha\wedge D_2v.$
In degree one, $L|_{V^1}=0$ gives
\[
 d_\theta|_{\cV^1}=D_1,\qquad
 \ker D_1=\operatorname{im}(d_\theta|_{\cV^0})=0.
\]
In degree four, $\cV^5=0$ gives
\[
 \alpha\wedge D_2(V^2)=\cV^4=\alpha\wedge V^3,
 \qquad D_2(V^2)=V^3.
\]
Thus $D_1$ is injective and $D_2$ is surjective. Also
\[
 D_2D_1=WB|_{V^1}=0.
\]
Consequently
\[
 1=\dim D_1(V^1)\leq\dim\ker D_2=m-1.
\]
The nonzero root $c$ occupies at least one dimension of the rank-four
middle cut. Therefore $2\leq m\leq3.$
Since $c\ne0$ and $t=0$, the zero multiplicities satisfy
\[
 m=\operatorname{ord}_{s=0}P_2
   =1+\operatorname{ord}_{s=0}Q.
\]
It follows that $Q(0)=e_2=0$, so $Q(s)=s(s-e_1)$, proving \eqref{two}.
\end{proof}

\section{Fixed-cut torsion and the ordinary value}\label{sec:normalization}
We write $Z(s,z,\tau)=Z_{g_\tau,\chi_z}(s)$ and denote by
$G(s,z,\tau)$ the complementary factor in
\cite[(6.4)--(6.5)]{CD}. This factor accounts for the part of the zeta function
outside the fixed resonance cluster. It is holomorphic and
nonvanishing in the spectral disk for each parameter.
The full cut in degree $k$ is
$\cU^k\oplus\alpha_\tau\wedge\cU^{k-1}$, with determinant $P_kP_{k-1}$
and we set $P_{-1}=P_5=1$. The contact Lefschetz isomorphisms give $P_3=P_1,$ and $P_4=P_0.$
We define the weighted determinant by
\begin{equation}\label{fulltelescoping}
 \mathfrak D_{\rm cut}(s)
 =\prod_{k=0}^{5}(P_kP_{k-1})^{(-1)^k k}
 =\frac{P_1^2}{P_0^2P_2}
 =\frac{s-t}{(s-c)Q(s)}.
\end{equation}
The two full degrees containing $P_j$ give total exponent
$(-1)^j j+(-1)^{j+1}(j+1)=(-1)^{j+1}$.
In flow dimension five, \cite[(6.5)]{CD} gives $G=Z\mathfrak D_{\rm cut}.$
Therefore
\begin{equation}\label{divisor}
 Z(s,z,\tau)=G(s,z,\tau)
       \frac{(s-c)(s^2-e_1s+e_2)}{s-t}.
\end{equation}
A dot denotes differentiation with respect to $\tau$. For a
degree-preserving map $T$ of a finite graded space $C^\bullet$, we define
the alternating trace, or supertrace, by
\[
 \str_{C^\bullet}T=\sum_k(-1)^k\tr_{C^k}T.
\]
\begin{lemma}\label{cutvariation}
Fix small real $\theta\ne0$. Let $\Pi_{\theta,\tau}$ be any full
resonance cut whose fixed contour remains free of resonances throughout
the parameter interval, and set $C^\bullet_{\theta,\tau}=\ran\Pi_{\theta,\tau}$.
We write $\mathcal T_{\rm fin}$ for its contact-chirality torsion
\cite[Section~3.3]{CD} and define
\[
 H(\theta,\tau)=\str_{C^\bullet_{\theta,\tau}}
       (\Pi_{\theta,\tau}\alpha_\tau\wedge\iota_{\dot X_\tau}).
\]
Then
\[
 \partial_\tau\log\mathcal T_{\rm fin}(\tau)=-H(\theta,\tau).
\]
Nonzero resonances inside the cut are allowed.
For the cut of Section~\ref{sec:spectral}, we also have
\begin{equation}\label{Gvariation}
 \frac{G(0,\theta,\tau)}{G(0,\theta,0)}
 =\exp\left(\int_0^\tau H(\theta,v)\,dv\right).
\end{equation}
For this cut, the finite trace $H$ extends continuously and boundedly
to $\theta=0$. In particular the ratio tends to one uniformly as
$\tau\to0$ and $\theta\to0$ through nonzero real characters.
\end{lemma}
\begin{proof}
We fix $\theta$ and write
$C_\tau=C^\bullet_{\theta,\tau}$ and
$U^k_\tau=C^k_\tau\cap\ker\iota_{X_\tau}$.
We suppress metric subscripts and use Chaubet--Dang's convention
$v\wedge\alpha=(-1)^{k-1}\alpha\wedge v$ for horizontal $(k-1)$-forms.
For a horizontal form $u$, the contact splitting and Lefschetz map are
\[
 C^k_\tau=U^{k-1}_\tau\wedge\alpha_\tau\oplus U^k_\tau,
 \qquad \ell_\tau u=u\wedge d\alpha_\tau,
 \qquad U^{-1}_\tau=U^5_\tau=0.
\]
For $k=0,1,2$, $f\in U^{k-1}$, and $g\in U^k$, we define the contact
chirality $\Gamma$, a degree-reversing involution, by
\begin{equation}\label{xp:contact-chirality}
 \Gamma(f\wedge\alpha+g)
 =\ell^{2-k}g\wedge\alpha+\ell^{3-k}f.
\end{equation}
In higher degrees it is the inverse. Thus
\[
 \Gamma(C^k_\tau)=C^{5-k}_\tau,\qquad\Gamma^2=\Id.
\]
The powers of $\ell$ denote the corresponding Lefschetz maps.
The inverses used below are restricted to the ranges of these maps.

The generalized zero-resonance part is acyclic by
\cite[Theorem~2.1]{DR}. On the sum of the
nonzero generalized resonance spaces, $L$ is invertible and
\[
 d_\theta(\iota_XL^{-1})+(\iota_XL^{-1})d_\theta=\Id.
\]
Thus the whole cut $C_\tau$ is acyclic. The spectral projector commutes with $d_\theta$,
$\iota_X$, wedging by $\alpha$, and the Lefschetz maps.
Thus $\Gamma$ preserves the full resonance space.

Write $\det E$ for the top exterior power of a finite-dimensional space.
The determinant line of the complex is
\[
 \det C_\tau=\bigotimes_{k=0}^5(\det C^k_\tau)^{(-1)^k},
\]
where exponent $-1$ denotes the dual line. Choose volume elements in
degrees $0,1,2$ and their $\Gamma$-partners in degrees $5,4,3$.
Their alternating tensor product, with the sign of~\cite[Section~3.3]{CD},
is choice-independent: paired rescalings cancel. By acyclicity, the canonical torsion
isomorphism takes values in $\det H^\bullet(C_\tau)=\mathbb C$.
It sends this tensor to the nonzero scalar $\mathcal T_{\rm fin}(\tau)$.
The differentiability of the spectral projectors established in
\cite[Proposition~7.2 and Appendix~B]{CD} implies that the finite-dimensional
matrices depend differentiably on $\tau$. Here the projectors act from a
stronger anisotropic space to a weaker one.
We take a reference time $\tau_*$ and write
\[
 \Pi_*=\Pi_{\theta,\tau_*},\qquad
 \dot\Pi_*=(\partial_\tau\Pi_{\theta,\tau})|_{\tau=\tau_*}.
\]
On the fixed space $C=C_{\tau_*}$, set
\[
 S_\tau=\Pi_{\theta,\tau}|_C,
 \qquad \mathsf F_\tau=\Pi_*S_\tau:C\to C.
\]
They are invertible near $\tau_*$. Since the projections commute
with $d_\theta$, the differential in the fixed coordinates is unchanged $S_\tau^{-1}d_\theta S_\tau=d_\theta|_C.$
We define the chirality in these coordinates by
\[
 \widetilde\Gamma_\tau
 =S_\tau^{-1}\Gamma_\tau S_\tau
 =\mathsf F_\tau^{-1}\Pi_*\Gamma_\tau S_\tau.
\]
At $\tau_*$, differentiation of $\Pi_{\theta,\tau}^2=\Pi_{\theta,\tau}$ gives
\[
 \Pi_*\dot\Pi_*\Pi_*=0,
 \qquad \mathsf F_{\tau_*}=\Id,\qquad\dot{\mathsf F}_{\tau_*}=0.
\]
Using $\Pi_*\Gamma_{\tau_*}=\Gamma_{\tau_*}\Pi_*$, we obtain
\[
 \dot{\widetilde\Gamma}_{\tau_*}
 =\Pi_*\dot\Gamma_{\tau_*}|_C
   +\Pi_*\Gamma_{\tau_*}\dot\Pi_*|_C
 =\Pi_*\dot\Gamma_{\tau_*}|_C.
\]
Since the differential does not change in these coordinates,
neither does the torsion isomorphism.
Fix volumes $\upsilon_k\in\det C^k$, $k=0,1,2$.
We differentiate the factors
$(\widetilde\Gamma_\tau\upsilon_k)^{(-1)^{5-k}}$
and use the derivative of $\widetilde\Gamma^2=\Id$.
Pairing the terms in complementary degrees gives the
half-supertrace identity~\cite[Proposition~4.9]{BK}, also used in
\cite[Lemma~7.3]{CD}:
\begin{equation}\label{xp:chirality-derivative}
 \begin{aligned}
 \partial_\tau\log\mathcal T_{\rm fin}(\tau)
 &=\sum_{k=0}^{2}(-1)^{k+1}
       \tr_{C^k}(\widetilde\Gamma\dot{\widetilde\Gamma})=\tfrac12\str_C(\dot{\widetilde\Gamma}\widetilde\Gamma)
   =\sum_{k=0}^{2}(-1)^{k+1}\tr_{C^k}(\Pi\Gamma\dot\Gamma).
 \end{aligned}
\end{equation}

We define a real function $h$ on $M$, a horizontal two-form $\delta$, and maps
$\mathsf A_0,\mathsf A_1$ on horizontal degrees zero and one by
\begin{equation}\label{xp:chirality-bundle-maps}
 \begin{aligned}
 h&=\alpha(\dot X)=-\dot\alpha(X),
 &\delta&=d\dot\alpha+(\iota_Xd\dot\alpha)\wedge\alpha,\\
 \mathsf A_0u&=2(\ell^2)^{-1}(u\wedge\delta\wedge d\alpha),
 &\mathsf A_1u&=\ell^{-1}(u\wedge\delta).
 \end{aligned}
\end{equation}
Here $\Pi h$ denotes multiplication by $h$ followed by projection to the cut.
To differentiate the chirality, we use
\[
 \partial_\tau(d\alpha_\tau)^j
 =j\,d\dot\alpha_\tau\wedge(d\alpha_\tau)^{j-1},
 \qquad\dot\alpha(X)=-h.
\]
In degree zero, for a fixed scalar $u$,
\[
 \begin{aligned}
 \dot\Gamma u
 &=2u\,d\dot\alpha\wedge d\alpha\wedge\alpha
      +u(d\alpha)^2\wedge\dot\alpha \text{ and }
 \Gamma\dot\Gamma u=\mathsf A_0u-hu.
 \end{aligned}
\]
The same differentiation in degrees one and two gives the diagonal
blocks of $\Gamma\dot\Gamma$:
\[
 \begin{array}{c|cc}
 k&U^{k-1}\wedge\alpha&U^k\\ \hline
 0&0&\mathsf A_0-h\\
 1&\mathsf A_0+h&\mathsf A_1-h\\
 2&\mathsf A_1+h&-h
 \end{array}
\]
The off-diagonal blocks do not contribute to the trace after projection
to the resonance space.
Set
\[
 a_j=\tr_{U^j}(\Pi\mathsf A_j)\quad(j=0,1),\qquad
 h_j=\tr_{U^j}(\Pi h)\quad(0\leq j\leq4),\qquad
 \gamma_k=\tr_{C^k}(\Pi\Gamma\dot\Gamma).
\]
The $\mathsf A_j$ contributions cancel between adjacent full degrees:
\begin{equation}\label{xp:chirality-trace-cancellation}
 \begin{aligned}
 \gamma_0&=a_0-h_0,\quad
 \gamma_1=a_0+h_0+a_1-h_1,\quad
 \gamma_2=a_1+h_1-h_2,\\
 -\gamma_0+\gamma_1-\gamma_2&=2h_0-2h_1+h_2.
 \end{aligned}
\end{equation}
For horizontal forms $v$ and $u$, contraction gives
\[
 (\alpha\wedge\iota_{\dot X})(v\wedge\alpha)=h\,v\wedge\alpha,
 \qquad (\alpha\wedge\iota_{\dot X})u
       =\alpha\wedge\iota_{\dot X}u.
\]
The second expression has no horizontal component. Thus the diagonal
of this operator is $h$ on $U^{k-1}\wedge\alpha$ and zero on $U^k$.
Hence
\begin{equation}\label{xp:full-cut-supertrace}
 H=-h_0+h_1-h_2+h_3-h_4=-2h_0+2h_1-h_2.
\end{equation}
The last equality follows from the cut Lefschetz isomorphisms and the
fact that scalar multiplication commutes with them:
\[
 h_3=h_1,\qquad h_4=h_0.
\]
Therefore
\begin{equation}\label{torsionvariation}
 \partial_\tau\log\mathcal T_{\rm fin}(\tau)=-H.
\end{equation}
For the cut of Section~\ref{sec:spectral}, the full ranks in \eqref{ranks}
and flow dimension five give the
sign exponent in \cite[(6.6)]{CD}:
\[
 n_{\rm cut}=\sum_{k=0}^2(-1)^k(3-k)\dim C^k
            =3-6+6=3.
\]
We write $\mathcal T_{\rm dyn}$ for Chaubet--Dang's dynamical torsion.
The finite-cut formula is thus
\begin{equation}\label{dynamicalbridge}
 \mathcal T_{\rm dyn}(\theta,\tau)
 =(-1)^3G(0,\theta,\tau)\mathcal T_{\rm fin}(\tau).
\end{equation}
At a crossing, the zero-complex sign exponent is $m-1$ and
$\operatorname{ord}_{s=0}Z=m-2$. With $\mathcal T_0$ the refined torsion
of the full generalized zero complex, \cite[(6.7)]{CD} gives
$\mathcal T_{\rm dyn}=-Z(0)\mathcal T_0$ for $m=2$, and
$\mathcal T_{\rm dyn}=Z'(0)\mathcal T_0$ for $m=3$, where $Z(0)=0$.
Acyclicity alone does not remove this finite factor. Formula~(6.8)
of~\cite{CD} applies when zero is absent from the resonance spectrum.
Deformation invariance~\cite[Theorem~9]{CD} at the fixed acyclic character and
\eqref{torsionvariation} give
\[
 0=\partial_\tau\log\mathcal T_{\rm dyn}
   =\partial_\tau\log G+\partial_\tau\log\mathcal T_{\rm fin}
   =\partial_\tau\log G-H.
\]
Integration proves \eqref{Gvariation}, without choosing a logarithm.

The dimensions of the full resonance spaces remain constant at the
nonacyclic trivial character.
To extend $H$ through $\theta=0$, set
$\mathcal K_\tau=\alpha_\tau\wedge\iota_{\dot X_\tau}$.
Using the forms $f_j$, functionals $\varphi_i$, and matrix $F$ of
\eqref{xp:frame-matrices}, define in each full degree
\[
 K^{\rm var}_{ij}(z,\tau)=\varphi_i
       (\Pi_{z,\tau}\mathcal K_\tau\Pi_{z,\tau}f_j).
\]
The matrices $F$ are uniformly invertible on $\mathcal P$.
The entries of $K^{\rm var}$ depend continuously on the parameters,
as follows from this composition of operators:
\[
 \mathcal H_2\xrightarrow{\Pi}\mathcal H_1
 \xrightarrow{\mathcal K_\tau}\mathcal H_1
 \xrightarrow{\Pi}\mathcal H_0
 \xrightarrow{\varphi_i}\mathbb C.
\]
Writing $F_k,K^{\rm var}_k$ for the matrices in degree $k$, we obtain the
continuous extension
\[
 H(z,\tau)=\sum_{k=0}^5(-1)^k
          \tr(F_k(z,\tau)^{-1}K^{\rm var}_k(z,\tau)).
\]
Choose $K_H\geq\sup_{\mathcal P}|H|$. Equation \eqref{Gvariation}
then yields
\[
 \left|\frac{G(0,\theta,\tau)}{G(0,\theta,0)}-1\right|
 \leq e^{K_H\tau}-1.
\]
\end{proof}

\begin{proof}[Proof of Theorem~\ref{mainresult}]
At the hyperbolic endpoint, \eqref{hyper}, \eqref{qhyper}, and
\eqref{divisor} give
\begin{equation}\label{hypervalue}
 Z_{\rm hyp}(\theta):=Z(0,\theta,0)
       =-G(0,\theta,0)c_0(\theta)^2\ne0.
\end{equation}
We fix a crossing $(\theta,\sigma_\theta)$ from Proposition~\ref{prop:crossing}.
Equation \eqref{two} gives $Q(s)=s(s-e_1)$. We first specialize the
parameters and then cancel the common $s$ in
\eqref{divisor}, without choosing root branches~\cite{Lawton}, to obtain
\begin{equation}\label{specialized}
 \begin{aligned}
 Z(s,\theta,\sigma_\theta)
 &=G(s,\theta,\sigma_\theta)(s-c)(s-e_1),\\
 Z(0,\theta,\sigma_\theta)&=G(0,\theta,\sigma_\theta)c e_1.
 \end{aligned}
\end{equation}
Here $c,e_1$ are evaluated at the crossing. If $m=3$, this has
a simple zero at zero. If $m=2$, its value there is nonzero.
In either case it is holomorphic and
\begin{equation}\label{exactratio}
 \frac{Z(0,\theta,\sigma_\theta)}{Z_{\rm hyp}(\theta)}
 =-\frac{G(0,\theta,\sigma_\theta)}{G(0,\theta,0)}
   \frac{c(\theta,\sigma_\theta)}{c_0(\theta)}
   \frac{e_1(\theta,\sigma_\theta)}{c_0(\theta)}.
\end{equation}
The first factor is $1+O(\theta^2)$ by Lemma~\ref{cutvariation} and
$\sigma_\theta=O(\theta^2)$. The bounded metric derivative gives
\[
 \frac{C_c(\theta,\sigma_\theta)}{C_c(\theta,0)}-1
 =\frac{\sigma_\theta}{C_c(\theta,0)}
   \int_0^1\partial_\tau C_c(\theta,v\sigma_\theta)\,dv
 =O(\theta^2).
\]
Together with \eqref{twoaxisfactor}, this proves
\begin{equation}\label{coefficientratios}
 \begin{aligned}
 \frac{c(\theta,\sigma_\theta)}{c_0(\theta)}
 &=\frac{C_c(\theta,\sigma_\theta)}{C_c(\theta,0)}=1+O(\theta^2) \text{ and }
 \frac{e_1(\theta,\sigma_\theta)}{c_0(\theta)}
 =\frac{\sigma_\theta E(\theta,\sigma_\theta)}{C_c(\theta,0)}
   =O(\theta^2).
 \end{aligned}
\end{equation}
Dividing the second ratio by $\theta^2$ gives
\[
 \frac{e_1(\theta,\sigma_\theta)}{\theta^2c_0(\theta)}
 =\frac{\sigma_\theta}{\theta^2}
   \frac{E(\theta,\sigma_\theta)}{C_c(\theta,0)}
 \longrightarrow\frac\kappa{2r}\frac{E(0,0)}{-\kappa/2}
 =-\frac{E(0,0)}r.
\]
The minus sign in \eqref{exactratio} therefore gives
\begin{equation}\label{quantitativeratio}
 \frac{Z(0,\theta,\sigma_\theta)}{Z_{\rm hyp}(\theta)}
 =\frac{E(0,0)}{r}\theta^2+o(\theta^2).
\end{equation}
We do not evaluate $E(0,0)$. If it vanishes, the ratio is
$o(\theta^2)$. This does not determine whether
$e_1(\theta,\sigma_\theta)=0$.
Fried's hyperbolic theorem \cite{Fried}, in the complex unitary
form stated in \cite[Theorem 3]{DGRS}, gives $|Z_{\rm hyp}(\theta)|=T_{\chi_\theta}(\Sigma)^2.$
In particular, the torsion-normalized value satisfies
\begin{equation}\label{quantitativebound}
 \frac{|Z_{g_{\sigma_\theta},\chi_\theta}(0)|}
      {T_{\chi_\theta}(\Sigma)^2}=O(\theta^2).
\end{equation}
For fixed sufficiently small real $\theta\ne0$ and
$\tau\in[0,\varepsilon]\setminus\{\sigma_\theta\}$, we have $ct\ne0$.
Equation~\eqref{P1} and the contact splitting show that zero is not
a resonance on full degree-one forms. Since the pulled-back character
is acyclic, \cite[Proposition~7.3]{DGRS} then implies that zero is not
a resonance in any degree. Thus \cite[Theorem~9 and (6.8)]{CD}
give $Z(0,\theta,\tau)=Z_{\rm hyp}(\theta)$, proving
\eqref{Friedidentity} at every other metric in this small family.
The acyclic unitary analytic torsion is metric-independent.
Choose sufficiently small real $\theta_j$ and set
\[
 0\ne\theta_j\longrightarrow0,\qquad
 g_j=g_{\sigma_{\theta_j}},\qquad\chi_j=\chi_{\theta_j}.
\]
Each metric is real analytic in the hyperbolic atlas of $g_0$,
since the fixed conformal function $b$ is real analytic.
The metrics converge smoothly because
$\sigma_{\theta_j}\to0$.
\end{proof}

\appendix
\section{Geometric and spectral inputs}\label{app:inputs}

\paragraph{The base and its characters.}
Write $b_k=b_k(\Sigma)$ and set
\[
 A_{\rm hom}=\begin{pmatrix}2&1\\1&-1\end{pmatrix},\qquad
 T_{\rm hom}=\begin{pmatrix}A_{\rm hom}&-\Id_2\\\Id_2&0\end{pmatrix}.
\]
The symmetry of $A_{\rm hom}$ makes $T_{\rm hom}$ preserve the standard
symplectic form. Its characteristic polynomial is
\[
 \begin{aligned}
 p(x)&:=\det(x\Id_4-T_{\rm hom})
 =\det\bigl((x^2+1)\Id_2-xA_{\rm hom}\bigr)\\
 &=(x^2-2x+1)(x^2+x+1)-x^2
 =x^4-x^3-x^2-x+1.
 \end{aligned}
\]
We write $\Phi_5$ for the fifth cyclotomic polynomial. Reduction modulo
two gives
\[
 p(x)\bmod2=x^4+x^3+x^2+x+1=\Phi_5(x)\bmod2.
\]
This polynomial is irreducible over the field with two elements because
the order of $2$ modulo $5$ is $4$. Hence $p$ is irreducible over
$\mathbb Q$. Also $p(1)=-1$ and $ p(2)=3.$
Thus $p$ has a real root greater than one. Since $p$ is irreducible,
a root of unity would force $p$ to equal its cyclotomic minimal polynomial,
contradicting that real root. Hence $p$ has no roots of unity as zeros.
Finally, for any polynomial
$q$ and integer $k>1$,
\[
 p'(0)=-1,\qquad
 \left.\frac{d}{dx}q(x^k)\right|_{x=0}=0.
\]
Therefore $p(x)$ cannot equal $q(x^k)$.
Surjectivity of the genus-two mapping-class action onto the integral
symplectic group, together with \cite[Proposition~1]{MS}, therefore gives
an orientation-preserving
pseudo-Anosov mapping class inducing $T_{\rm hom}$.
Its mapping torus $\Sigma$ is hyperbolic by \cite{Thurston}.
The Wang sequence of the mapping torus gives
\[
 b_1(\Sigma)=1+\dim\ker(T_{\rm hom}-\Id_4)=1,
\]
because $p(1)=-1$ makes $T_{\rm hom}-\Id_4$ invertible.
We choose a hyperbolic metric $g_0$ and a nonzero real harmonic
one-form $\omega$.
The hyperbolic charts give $\Sigma$ a real analytic atlas.

All norms below are $L^2(g_0)$ norms.
We choose $\eta>0$ below the first positive untwisted scalar and
one-form eigenvalues.
The scalar and one-form spectral projections onto $[0,\eta)$ both have
rank one for small $\theta$.
A loop with nonzero $\omega$-period proves that $\chi_\theta$ is
nontrivial for all sufficiently small nonzero $\theta$.
The scalar eigenvalue $\mu_0(\theta)$ in this one-dimensional space is then positive:
a harmonic scalar section would be parallel.
For a normalized eigenfunction $f_\theta$,
\[
 \|d_\theta f_\theta\|_2^2=\mu_0(\theta)>0,\qquad
 \Delta_{\mathrm H,\theta}^{(1)}d_\theta f_\theta
   =\mu_0(\theta)d_\theta f_\theta.
\]
Since the one-form spectral space below $\eta$ has dimension one,
$d_\theta f_\theta$ spans it. Its eigenvalue is positive, so
\[
 \ker\Delta_{\mathrm H,\theta}^{(0)}
 =\ker\Delta_{\mathrm H,\theta}^{(1)}=0,
 \qquad H^0(\Sigma;\chi_\theta)=H^1(\Sigma;\chi_\theta)=0.
\]
Unitary Hodge duality then gives $H^k(\Sigma;\chi_\theta)=0$ for $0\leq k\leq3$.
The spectral sequence of the oriented $S^2$-bundle $M\to\Sigma$ has terms $H^p(\Sigma;\chi_\theta)\otimes H^q(S^2;\mathbb C)=0.$
It follows that $H^\bullet(M;\pi^*\chi_\theta)=0.$
For small nonzero real $\theta$, every form in the small one-form eigenspace
is exact. The same argument therefore gives a uniform gap $\eta$ on coclosed
one-forms.
At $\theta=0$, this gap holds after removing the harmonic
one-form space.

Set $V_\Sigma=\operatorname{vol}_{g_0}(\Sigma)$ and define scalar operators
$K_1,K_2$ by
$K_1f=-2i\langle\omega,df\rangle$ and $K_2f=|\omega|^2f$
for smooth $f$, with the pointwise $g_0$ inner product.
The harmonicity of $\omega$ gives $\Delta_{\mathrm H,\theta}^{(0)}
 =\Delta_{\mathrm H,0}^{(0)}+\theta K_1+\theta^2K_2.$
We take the normalized constant $f_0=V_\Sigma^{-1/2}$. Then $K_1f_0=0$ and $\mu_0'(0)=0.$
Differentiating the scalar eigenvalue equation at zero and normalizing
$f_\theta$ so that $f'_0$ is orthogonal to $f_0$ gives
\[
 \Delta_{\mathrm H,0}^{(0)}f'_0=-K_1f_0+\mu_0'(0)f_0=0,
 \qquad f'_0=0.
\]
The second derivative, therefore, gives
\[
 \frac{\mu_0''(0)}2
 =\langle f_0,K_2f_0\rangle_{L^2}
 =\frac1{V_\Sigma}\int_\Sigma|\omega|^2\,d\operatorname{vol}_{g_0}
 =\kappa.
\]
Simple-eigenvalue perturbation gives an analytic eigenvalue branch,
and complex conjugation gives $\mu_0(-\theta)=\mu_0(\theta)$. Thus
\begin{equation}\label{app:scalarquadratic}
 \mu_0(\theta)=\kappa\theta^2+O(\theta^4),\qquad
 c_0(\theta)=-\frac{\mu_0(\theta)}{1+\sqrt{1-\mu_0(\theta)}}
 =-\frac\kappa2\theta^2+O(\theta^4).
\end{equation}

\paragraph{The conformal sign.}
Here $z=\tau=0$, so $X=X_0$, $\alpha=\alpha_0$, and $\Delta_{g_0}$ is the
nonpositive scalar Laplacian used in \cite{CDDP}.
In the Levi--Civita splitting of $TM$, we write tangent vectors as
$\xi=(\xi_H,\xi_V)$ and $\eta=(\eta_H,\eta_V)$, with horizontal
and vertical components. The cross product uses the metric and
orientation on $\Sigma$. We define the tangent rotation and two-form
\[
 \mathsf I_{(x,v)}(\xi_H,\xi_V)
   =(v\times\xi_V,v\times\xi_H),\qquad
 \psi(\xi,\eta)=d\alpha(\mathsf I\xi,\eta).
\]
We write $J(x,v)=(x,-v)$ for time reversal and use $\pi_*$ for
fiber integration with
the convention of \cite{CDDP}.
Let $\cC$ be the space of closed horizontal resonant one-forms at zero and
$\cC_\psi=\{v\circ\mathsf I:v\in\cC\}$ its complementary summand
from~\cite[Lemma~3.9]{CDDP}. The pushforward isomorphism in
\cite[Lemma~3.17]{CDDP} commutes with conjugation. Its inverse image
$u\in\cC_\psi$ of the real form $\omega$ is therefore real and satisfies
$\pi_*(d\alpha\wedge u)=\omega$.
Then $du\ne0$ and $u\circ\mathsf I$ is closed.
Set $u_*=J^*u$ and $N_\omega=\|\omega\|_2^2$.
The results in~\cite[Lemma~3.17 and Corollary~5.1]{CDDP} identify
the pushforward of $u_*$ as $\pi_*(d\alpha\wedge u_*)=\omega.$
The outgoing and incoming wavefront conditions give $\operatorname{WF}(du)\cap\bigl(-\operatorname{WF}(du_*)\bigr)=\varnothing.$
Their product is therefore defined, as explained in~\cite[Section~2.3]{CDDP}.
Define a real distribution $F_{\rm dist}$ and a smooth function $f$ by
\[
 \pi_*(\alpha\wedge du\wedge du_*)
       =F_{\rm dist}\,d\operatorname{vol}_{g_0},\qquad
 F_{\rm dist}\in\mathcal D'(\Sigma;\mathbb R),\qquad
 f=\Delta_{g_0}|\omega|^2.
\]
For a function $h$ on $\Sigma$, let $\widetilde h$ be its lift to
hyperbolic three-space $\mathbb H^3$, with distance $d_{\mathbb H^3}$.
Define $Q_4$ by
\begin{equation}\label{app:Q4definition}
 \widetilde{Q_4h}(x)=\int_{\mathbb H^3}
 \frac{\widetilde h(y)}{\cosh^4 d_{\mathbb H^3}(x,y)}
       \,d\operatorname{vol}_{\mathbb H^3}(y).
\end{equation}
The periodized kernel is smooth and symmetric on the compact quotient.
Thus $Q_4$ extends to distributions by duality and maps them to smooth functions.
By~\cite[Theorem~5 and Proposition~A.1]{CDDP}, we have $Q_4F_{\rm dist}=-f/6,$ with $ f\ne0.$
We take $b=Q_4f$ as our conformal direction.
This function is real analytic. Indeed, after a hyperbolic isometry,
any point can be taken to $(1,0)$ in the hyperboloid model.
Vector norms are Euclidean, and the dot denotes the complex bilinear product.
The model is
\[
 \mathbb H^3=\{(y_0,y')\in\mathbb R\times\mathbb R^3:
                  y_0>0,\ y_0^2-|y'|^2=1\}.
\]
For a complex point $w=(w_0,w')\in\mathbb C\times\mathbb C^3$
near $(1,0)$, the kernel in
\eqref{app:Q4definition} extends to $(w_0y_0-w'\cdot y')^{-4}$.
Since $|y'|\leq y_0$, the condition
$|w_0-1|+|w'|<1/2$ gives
\[
 \begin{aligned}
 |w_0y_0-w'\cdot y'|
 &\geq y_0-|w_0-1|y_0-|w'||y'|\geq(1-|w_0-1|-|w'|)y_0>y_0/2.
 \end{aligned}
\]
The lifted function $\widetilde f$ is bounded. In polar coordinates,
with distance variable $\varrho$, we have $y_0=\cosh\varrho$ and
\[
 \int_{\mathbb H^3}y_0^{-4}\,d\operatorname{vol}_{\mathbb H^3}
 =4\pi\int_0^\infty\frac{\sinh^2\varrho}{\cosh^4\varrho}\,d\varrho
 =\frac{4\pi}{3}<\infty.
\]
The integral therefore extends holomorphically near $(1,0)$.
Applying this at every point shows that $b$, and hence
$g_\tau=e^{-2\tau b}g_0$, is real analytic in the hyperbolic atlas of $g_0$.

The fiber orientation used to define $\pi_*$, together with the
orientation of $\Sigma$, induces the opposite of the contact
orientation on $M$, as shown in~\cite[(2.20)--(2.21)]{CDDP}. Thus, for a
top-degree form $\beta$ and a smooth function $h$ on $\Sigma$,
\[
 \int_M\pi^*h\,\beta=-\int_\Sigma h\,\pi_*\beta,
 \qquad
 \pi_*(\alpha\wedge(d\alpha)^2)=-8\pi\,d\operatorname{vol}_{g_0}.
\]
The integration identity also holds for distributional top-degree
forms, interpreted by duality.
With distributional pairings taken against $d\operatorname{vol}_{g_0}$,
the identity in~\cite[(5.6)]{CDDP} gives
\[
 \begin{aligned}
 S&:=\int_M\pi^*b\,\alpha\wedge du\wedge du_*=-\langle F_{\rm dist},Q_4f\rangle
   =-\langle Q_4F_{\rm dist},f\rangle
   =\frac16\|f\|_2^2>0.
 \end{aligned}
\]
For outgoing and incoming one-forms $v,v_*$, we define the pairing
$\mathcal B$ by
$\mathcal B(v,v_*)=\int_M\alpha\wedge d\alpha\wedge v\wedge v_*$.
We set $u_{\rm cl}=u\circ\mathsf I$ and use brackets for de Rham
classes. The identities in
\cite[Lemmas~2.10 and 3.5, (3.57), (3.67), and the calculation preceding (3.59)]{CDDP}
give
\[
 [u_{\rm cl}]=-[\pi^*\omega]/(4\pi),\qquad
 \mathcal B(u_{\rm cl},J^*u_{\rm cl})=-N_\omega/(4\pi),\qquad
 \mathcal B(u,u_*)=N_\omega/(4\pi).
\]
The same identities give $\pi_*(\psi\wedge u_{\rm cl})=\omega$.
The closed and nonclosed degree-one blocks each have dimension one,
and their pairing is diagonal by~\cite[(3.55)--(3.57)]{CDDP}.
We write $P'_{\rm cut}$ and $B'_{\rm cut}$ for the metric derivatives
at zero compressed to the degree-one cut. Here $P=-iL=iB$.
The identification $F_\tau$ of the unit tangent bundles gives
\[
 \alpha_\tau=e^{-\tau\pi^*b}\alpha_0,
 \qquad\iota_X\partial_\tau\alpha_\tau|_0=-b\circ\pi.
\]
The identities in~\cite[Lemma~4.4, (4.22), and (4.38)]{CDDP} therefore give $\mathcal B(P'_{\rm cut}u,u_*)=iS.$
In the ordered basis $(u_{\rm cl},u)$, the entries in the closed row
and column vanish because $du_{\rm cl}=d(J^*u_{\rm cl})=0$.
The remaining diagonal entry is
\[
 \frac{\mathcal B(P'_{\rm cut}u,u_*)}{\mathcal B(u,u_*)}
 =\frac{iS}{N_\omega/(4\pi)}
 =i\,\frac{4\pi S}{N_\omega}
 =i\,\frac{2\pi}{3N_\omega}\|f\|_2^2.
\]
Consequently
\begin{equation}\label{app:positivevariation}
 P'_{\rm cut}=\operatorname{diag}(0,ir),\qquad
 B'_{\rm cut}=\operatorname{diag}(0,r),\qquad
 r=\frac{2\pi}{3N_\omega}\|f\|_2^2>0.
\end{equation}
Reversing $b$ reverses $S$ and $B'_{\rm cut}$.
Constants span the rank-one scalar cut, so $c(0,\tau)=0$.
The finite-cut trace identity~\cite[Lemma~A.2]{CD} gives
$\partial_\tau t(0,0)=\tr B'_{\rm cut}=r$. Since $t(0,0)=0$,
this proves \eqref{taxis}.
For $b_1=1$, this is the nondegeneracy condition
in~\cite[Theorem~3]{CDDP}. Thus
\cite[Theorem~3 and Corollary~4.1]{CDDP} give $\varepsilon_0>0$ such that
the horizontal untwisted generalized zero-resonance dimensions are $(1,1,3,1,1)$
for $0<\tau<\varepsilon_0$.

\paragraph{The hyperbolic clusters.}
For $l=0,1,2$ and integers $p\geq0$, we denote by $\nu_l$ and $\sigma_p$
the $SO(2)$ representations on
$\Lambda^l\mathbb C^2$ and trace-free symmetric tensors of order
$p$, respectively.
Here $\mathbf1$ is the trivial representation. Weight $p$ means that
rotation through $\varphi$ acts by $e^{ip\varphi}$.
We have
\[
 \nu_0=\nu_2=\sigma_0=\mathbf1,\qquad
 \sigma_1=\nu_1,\qquad
 \nu_1\otimes\nu_1=\mathbf1^{\oplus2}\oplus\sigma_2.
\]
The weights of $\sigma_1$ are $\pm1$, and those of $\sigma_2$ are $\pm2$.
For a complex variable $a$ and an $SO(2)$ representation $\nu$, we write
$Z^D_\nu(a)$ for the Selberg function in~\cite{DGRS}, with character
$\chi_\theta$. The sign $\doteq$ means equality up to a holomorphic
factor with no zeros near the spectral point in question.
Write $\Det_\zeta$ for the zeta-regularized determinant. Graded
determinants are numerator determinants divided by denominator determinants.
With the centered Selberg variable $q=a-1$, the graded determinant
formula in~\cite[Proposition~4.7]{BO} gives
\begin{align}\label{app:determinants}
 Z^D_{\mathbf1}(a)&\doteq
 \Det_\zeta(\Delta_{\mathrm H,\theta}^{(0)}-1+(a-1)^2),\\
 Z^D_{\nu_1}(1+q)&\doteq
 \frac{\Det_\zeta(\Delta_{\mathrm H,\theta}^{(1)}+q^2)}
      {\Det_\zeta(\Delta_{\mathrm H,\theta}^{(0)}+q^2)}.
\end{align}
Each weight pair occurs once.
The exact one-form spectrum cancels the scalar denominator.
At $\theta=0$, the harmonic contribution cancels as well, since
$b_0=b_1=1$.
The spectral gap on coclosed one-forms therefore makes the centered
weight-one factor holomorphic and nonvanishing on a fixed small disk.

The weight-two factor uses the formal difference
$\operatorname{Sym}^2_0\mathbb C^3-\Lambda^1\mathbb C^3$ of $SO(3)$
representations, with $\operatorname{Sym}^2_0$ denoting trace-free symmetric
tensors. We use the same notation for the associated bundles twisted by $\chi_\theta$.
Here $\nabla_\theta$ is their tensor connection, and $\nabla_\theta^*$ is its $L^2(g_0)$ adjoint.
Appendix~B of~\cite{DGRS} and the same determinant formula give, at $a=s+2$,
the numerator and denominator operators
\[
 \begin{gathered}
 \nabla_\theta^*\nabla_\theta-2+2s+s^2
       \quad\hbox{on }\operatorname{Sym}^2_0\mathbb C^3 \text{ and }
 \nabla_\theta^*\nabla_\theta+2+2s+s^2
       \quad\hbox{on }\Lambda^1\mathbb C^3.
 \end{gathered}
\]
The Bochner bound used in the proof of~\cite[Proposition~7.7]{DGRS} is
$\nabla_\theta^*\nabla_\theta\ge3$ on $\operatorname{Sym}^2_0$.
A flat unitary twist changes no curvature term.
At $s=0$ both operators are strictly positive. Adding the scalar
$2s+s^2$ keeps them invertible when $|2s+s^2|<1/2$.

\begin{samepage}
In~\cite[(7.6)]{DGRS}, the radial index is $j\ge0$, the exterior index
is $0\le l\le k$, and the symmetric degree is $p\ge0$.
The representation is $\nu_l\otimes\nu_{k-l}\otimes\sigma_p$
and its shift is $\mathfrak h=2(j-l)+p+2+k$.
For $k=0,1,2$, the complete list with $\mathfrak h\le2$ is
\[
\begin{array}{ccccl}
 k&l&p&j&\text{factor}\\\hline
 0&0&0&0&Z^D_{\mathbf1}(s+2)\\
 1&1&0&0&Z^D_{\nu_1}(s+1)\\
 1&1&1&0&Z^D_{\nu_1\otimes\nu_1}(s+2)\\
 2&2&0&0&Z^D_{\mathbf1}(s)\\
 2&2&0&1&Z^D_{\mathbf1}(s+2)\\
 2&2&1&0&Z^D_{\nu_1}(s+1)\\
 2&2&2&0&Z^D_{\sigma_2}(s+2)\\
 2&1&0&0&Z^D_{\nu_1\otimes\nu_1}(s+2).
\end{array}
\]
\end{samepage}
Write $\ell_*>0$ for the systole,
$q_*=e^{-\ell_*}$, and $\ell_\gamma=\ell_{g_0}(\gamma)$.
For fixed $k$, let $I_k$ be the triples $(j,l,p)$ with $\mathfrak h\ge3$, and put
$\mu_{klp}=\nu_l\otimes\nu_{k-l}\otimes\sigma_p$.
For $0<w<1$, the representation dimensions give
\[
 \sum_{j,p\ge0}\dim\sigma_p\,w^{2j+p}=(1-w)^{-2},
 \qquad\sum_{l=0}^k\dim\nu_l\dim\nu_{k-l}=\binom4k\le6.
\]
Since $2+k-2l\ge0$, for a real variable $t\ge\ell_*$ we obtain
\[
 \sum_{I_k}\dim\mu_{klp}\,e^{-\mathfrak h t}
 \le e^{-3(t-\ell_*)}\sum_{I_k}\dim\mu_{klp}\,e^{-\mathfrak h\ell_*}
 \le\frac{6e^{3\ell_*}}{(1-q_*)^2}e^{-3t}.
\]
For $n\ge1$, we denote the contracting determinant in
\cite[(7.5)]{DGRS} by $d_{\gamma,n}$. It satisfies
\[
 (1-q_*)^2\le d_{\gamma,n}\le(1+e^{-n\ell_\gamma})^2\le4.
\]
For real $\theta$, the character $\chi_\theta$ is unitary.
The logarithms defined by the Euler series therefore satisfy, for $0<\delta_s<1/2$,
\[
 \sum_{I_k}\sup_{\substack{\operatorname{Re}s\ge-\delta_s\\\theta\in\mathbb R}}
 |\log Z^D_{\mu_{klp}}(s+\mathfrak h)|
 \le\frac{6e^{3\ell_*}}{(1-q_*)^4}
 \sum_{\gamma\in\mathcal G_{g_0}^\#}\sum_{n\ge1}
 \frac{e^{-(3-\delta_s)n\ell_\gamma}}{n}<\infty.
\]
The last sum is bounded by four times the convergent scalar Euler
logarithm at $3-\delta_s>2$~\cite[Section~7.3]{DGRS}.
Exponentiating this uniformly convergent sum gives a holomorphic,
nonvanishing function on $\operatorname{Re}s>-\delta_s$.
This half-plane overlaps $\operatorname{Re}s>2$, where
\cite[(7.6)]{DGRS} holds, so the factorization continues to zero.
For the horizontal dynamical determinants $\mathscr D_k$
of Lemma~\ref{holom}, the table therefore gives
\[
 \begin{aligned}
 \mathscr D_0&\doteq Z^D_{\mathbf1}(s+2),\\
 \mathscr D_1&\doteq
 Z^D_{\nu_1}(s+1)Z^D_{\mathbf1}(s+2)^2Z^D_{\sigma_2}(s+2)
 \doteq Z^D_{\mathbf1}(s+2)^2,\\
 \mathscr D_2&\doteq
 Z^D_{\mathbf1}(s)Z^D_{\mathbf1}(s+2)^3
 Z^D_{\nu_1}(s+1)Z^D_{\sigma_2}(s+2)^2\doteq Z^D_{\mathbf1}(s)Z^D_{\mathbf1}(s+2)^3.
 \end{aligned}
\]
The scalar factors give
\[
 \begin{aligned}
 \mu_0+s^2+2s&=(s-c_0)(s+2+c_0)\text{ and }
 \mu_0+s^2-2s=(s+c_0)(s-2-c_0).
 \end{aligned}
\]
Their simple roots near zero are $c_0$ and $-c_0$, respectively.
The other roots tend to $-2$ and $2$.
All other scalar eigenvalues stay separated from zero.
The multiplicity statement following (7.3) in~\cite{DGRS} identifies
the determinant orders with algebraic resonance multiplicities, giving
\begin{equation}\label{app:clusterpolynomials}
 P_0(s)=s-c_0,\qquad P_1(s)=(s-c_0)^2,\qquad
 P_2(s)=(s-c_0)^3(s+c_0).
\end{equation}
At $\theta=0$, \cite[Proposition~7.7]{DGRS} gives the same horizontal
ranks $(1,2,4,2,1)$, with the last two degrees supplied by contact
Lefschetz.

\end{document}